\documentclass{article}
\usepackage{graphicx} 
\usepackage{amsmath}
\usepackage{amsfonts}
\usepackage{amssymb,dsfont}
\usepackage{amsthm}
\usepackage{comment}
\usepackage{diagbox}
\usepackage{hyperref}
\usepackage{ulem}
\usepackage{algpseudocode}
\usepackage{bm}

\newtheorem{theorem}{Theorem}[section]
\newtheorem{corollary}{Corollary}[theorem]
\newtheorem{lemma}[theorem]{Lemma}
\newtheorem{proposition}[theorem]{Proposition}
\theoremstyle{definition}
\newtheorem{definition}{Definition}[section]
\theoremstyle{remark}

\theoremstyle{definition}
\newtheorem{example}{Example}[section]

\title{A generalization of the hexastix arrangement \newline to higher dimensions}
\author{Jan Kristian Haugland \\ \texttt{admin@neutreeko.net}}

\begin{document}

\maketitle

\section{Introduction}

Hexastix is an arrangement of non-overlapping infinite hexagonal prisms in four different directions that cover $\frac{3}{4}$ of space (confer \cite{conway}). We consider a possible generalization to $n$ dimensions, based on the permutohedral lattice $A^*_n$. The central lines of the generalized prisms are going to be oriented in $n+1$ different directions (parallel to the shortest non-zero vectors of $A^*_n$). The projection of the lines oriented in any direction along that direction to a hyperplane perpendicular to it is required to be a translation of the corresponding projection of $A^*_n$, and the minimal distance between lines oriented in any two given directions should be maximal. In Section 3 we prove that this is possible if $n$ is a prime power. In Section 4 we calculate the proportion of $n$-space that is covered for $n \in \{4, 5\}$ when the generalized prisms are maximized, and in Section 5 we briefly consider an alternative generalization.

\section{Preliminaries}

Predominantly following \cite{baek}, $H_d$ is the hyperplane $\{x \in \mathbb{R}^{d+1}|\sum x_i = 0\}$, $A_d = \mathbb{Z}^{d+1} \cap H_d$ and $A^*_d$ is the permutohedral lattice, consisting of all points $x \in H_d$ for which $x_i - x_j \in \mathbb{Z}$ for all $i, j \in \{1, \dots, d+1\}$. $S^i_d$ is the projection mapping $\mathbb{R}^d$ onto $\mathbb{R}^{d-1}$ by removing the $i$th coordinate, and $T_d$ is the projection mapping $\mathbb{R}^d$ onto $H_{d-1}$ along the vector $(1, \dotsc, 1)$.

Let $v^{(1)} = \left(\frac{n}{n+1}, -\frac{1}{n+1}, \dotsc, -\frac{1}{n+1}\right)$, $v^{(2)} = \left(-\frac{1}{n+1}, \frac{n}{n+1}, -\frac{1}{n+1}, \dotsc, -\frac{1}{n+1}\right)$, $\dotsc$, $v^{(n+1)} = \left(-\frac{1}{n+1}, \dotsc, -\frac{1}{n+1}, \frac{n}{n+1}\right)$, which together with their negatives are the shortest non-zero vectors of $A^*_n$. Note that the composite projection \newline $T_n \circ S^i_{n+1}$ maps $x \in H_n$ along $v^{(i)}$ to the hyperplane $x_i=0$ and then removes the zero at the $i$th coordinate, and that $$T_n \circ S^i_{n+1} \left( A_n \right) = T_n \circ S^i_{n+1} \left( A^*_n \right) = A^*_{n-1}$$ for any $i$. Thus with $n+1$ vectors $\{u^{(i)}\}_{i \leq n+1}$ such that $u^{(i)} \in H_n$, $u^{(i)}_i = 0$ for all $i$, we can let the central lines of the generalized prisms oriented in the $i$th direction be given by $L_i = u^{(i)} + A_n + \{t v^{(i)} | t \in \mathbb{R}\}$. We then have $$T_n \circ S^i_{n+1} \left( L_i \right) = S^i_{n+1} \left( u^{(i)} \right) + A^*_{n-1}$$

\begin{example}
For $n=3$, take $u^{(1)} = \left( 0, \frac{1}{4}, -\frac{1}{4}, 0 \right)$, $u^{(2)} = \left( -\frac{1}{4}, 0, \frac{1}{4}, 0 \right)$, $u^{(3)} = \left( \frac{1}{4}, -\frac{1}{4}, 0, 0 \right)$, $u^{(4)} = \left( 0, 0, 0, 0 \right)$. For each $i \in \{1, 2, 3, 4\}$, each line in $L_i$ can be encapsulated by an infinite hexagonal prism of inradius $\frac{\sqrt{2}}{8}$, yielding the hexastix arrangement. For example, the prism containing $\left( 0, 0, 0, 0 \right)$ is given by $\lvert x_1 - x_2 \rvert \leq \frac{1}{4}$, $\lvert x_1 - x_3 \rvert \leq \frac{1}{4}$ and $\lvert x_2 - x_3 \rvert \leq \frac{1}{4}$.
\end{example}

With $i \neq j$, let $$w = T_{n-1} \circ S^{\operatorname{min} \{i, j\}}_{n-1} \circ S^{\operatorname{max} \{i, j \}}_n \left( u^{(i)} - u^{(j)} \right)$$ The distance between a line in $L_i$ and a line in $L_j$ is the length of a vector of the form $w + P$ with $P \in A^*_{n-2}$. In other words, the minimal distance between some line in $L_i$ and some line in $L_j$ is $D_{n-2}(w)$, where $D_d(x)$ is defined as the distance from $x \in H_d$ to the nearest lattice point $y \in A^*_d$.

\begin{lemma}
The maximal value of $D_d(x)$ is $\sqrt{\frac{d(d+2)}{12(d+1)}}$, attained if and only if the coordinates of $x$ are uniformly distributed modulo 1.
\end{lemma}

\begin{proof}
Let $\sigma(d) = \frac{d(d+1)(d+2)}{3}$. Since $\sigma(0) = 0^2$, $\sigma(1) = 2 \cdot 1^2$ and \newline $\sigma(d) - \sigma(d-2) = 2d^2$, it follows by induction on $d$ that $\sum_{i=1}^{d+1} \left({2(i-1) - d}\right)^2 = \sigma(d)$.

The Voronoi cell of the origin in $A^*_d$ is a permutohedron with vertices given by the permutations of $\left( -\frac{d}{2(d+1)}, -\frac{d-2}{2(d+1)}, \dotsc, \frac{d}{2(d+1)}\right)$, and it follows that the maximal value of the distance to a lattice point is $$\sqrt{\sum_{i=1}^{d+1} \left(\frac{2(i-1) - d}{2(d+1)}\right)^2} = \frac{\sqrt{\sigma(d)}}{2(d+1)} = \sqrt{\frac{d(d+2)}{12(d+1)}}$$ If the coordinates of $x$ are not uniformly distributed modulo 1, then neither are the coordinates of any point in $x+A^*_d$. On the other hand, if the coordinates of $x$ are uniformly distributed modulo 1, then there exists a vertex $y$ of that permutohedron (actually, $d+1$ vertices) such that the residues of $x_i - y_i$ modulo 1 are the same for all $i$, and it is easily verified that $x-y \in A^*_d$.
\end{proof}

\begin{corollary}
If $i \neq j$, the minimal distance between some line in $L_i$ and some line in $L_j$ is at most $\operatorname{max} D_{n-2}(x) = \sqrt{\frac{n(n-2)}{12(n-1)}}$, attained if and only if the coordinates of $u^{(i)} - u^{(j)}$, other than the $i$th and the $j$th, are uniformly distributed modulo 1. {\hfill \ensuremath{\Box}}
\end{corollary}

Note that the minimal distance between two distinct \textit{parallel} lines is $$\sqrt{\left( 1 - \frac{1}{n} \right)^2 + (n-1) \frac{1}{n^2}} = \sqrt{1 - \frac{1}{n}}$$ which is smaller than  $\operatorname{max} D_{n-2}(x)$ when $n \geq 13$.

\begin{definition}
An \textit{extended cyclic permutation} maps $(y_1, \dotsc, y_d)$ to \newline $(y_{a + b}, y_{2a + b}, \dotsc, y_b)$ for some integers $a$, $b$ where $(a, d) = 1$ and the indices are considered modulo $d$.
\end{definition}

\noindent This will be applied when we consider the details for $n=5$ in Section 4.

\section{A family of optimal solutions}

Suppose $n$ is a prime power. Let $F=\{a_1, \dotsc, a_n\}$ be the field of order $n$, and let $\alpha$ be a primitive element of $F$. We define an $(n+1) \times (n+1)$ matrix $M$ as follows. If $i=j$ or $i=n+1$ or $j=n+1$, let $M_{i,j} = 0$. If $i \neq j$ and $i, j \leq n$, let $M_{i,j}$ be the discrete logarithm (between 0 and $n-2$ inclusive) of $a_i - a_j$ to the base $\alpha$.

\begin{proposition}
With $M$ defined as above, let $i \neq j$. Then the difference between the $i$th and the $j$th row of $M$, except for the $i$th and the $j$th column, contains all residues modulo $n-1$.
\end{proposition}

\begin{proof}
Each of the first $n$ rows, when the corresponding and the $(n+1)$th column is removed, contains the discrete logarithms of all nonzero field elements. Thus, the statement is true if $i$ or $j$ is $n+1$, and we can assume that $i, j \in \{1, \dotsc, n\}$ henceforth. In this case, the difference in the $(n+1)$th column is zero, and it remains to verify that if $r \neq s$ and $r, s \in \{1, \dotsc, n\} \backslash \{i, j\}$, then $M_{i, r} - M_{j, r}$ and $M_{i, s} - M_{j, s}$ are distinct and nonzero residues $(\operatorname{mod} n - 1)$. They are nonzero since $M_{i, r}$ and $M_{j, r}$ are discrete logarithms of distinct nonzero field elements. In order to verify that they are distinct, note that $M_{i, r} - M_{j, r}$ is congruent $(\operatorname{mod} n - 1)$ to the discrete logarithm of $\frac{a_i - a_r}{a_j - a_r}$. Therefore, it suffices that $$\frac{a_i - a_r}{a_j - a_r} \neq \frac{a_i - a_s}{a_j - a_s}$$ which is equivalent to $$\left( a_i - a_r \right) \left( a_j - a_s \right) \neq \left( a_i - a_s \right) \left( a_j - a_r \right) \Longleftrightarrow \left( a_i - a_j \right) \left( a_r - a_s \right) \neq 0$$
\end{proof}

If a constant is added to each element of any row or any column, the residues will remain uniformly distributed. Thus, we can take

\[
u^{(i)}_j=
\begin{cases}
    0 &\quad\text{if }i = j \vee i = n + 1 \vee j = n + 1\\
    \frac{M_{i, j} + \frac{1}{2}}{n-1} - \frac{1}{2} &\quad\text{otherwise}
\end{cases}
\]

\noindent and the condition of Cor. 2.1.1 is satisfied for each pair of distinct indices.

\section{Generalized prisms}

Let $\{u^{(i)}\}$ be given as above. In order to compute the proportion of $n$-space that is covered, we can compare the $(n-1)$-volumes of the Voronoi cells of $(0, \dotsc, 0)$ for two different sets of points. Without loss of generality, we focus on the generalized prisms directed along $v^{(n+1)}$. Take the set of points $x \in H_{n-1}$ for which $x_i=0$ and $(x_1, \dotsc, x_{i-1}, x_{i+1}, \dotsc, x_n)$ is of the form $P + Q$ where $P \in A^*_{n-2}$ and $Q$ is obtained by removing the $i$th and the last coordinate of $u^{(i)}$, for some $i$. These are the nearest points of each line in $L_i$ to the line $\{tv^{(n+1)}|t \in \mathbb{R}\} \subset L_{n+1}$, projected to the hyperplane $x_{n+1}=0$ by removing the last coordinate. The union of these sets over $1 \leq i \leq n$ is the first of the two sets mentioned above. Then the projection of each generalized prism along the central line is isometric to the Voronoi cell $V$ of $(0, \dotsc, 0)$. Denote the $(n-1)$-volume of $V$ by $\lambda(V)$. If we had used just one direction and filled $n$-space with the generalized prisms, their projections would be the Voronoi cells of $A^*_{n-1}$, which is the second set that we are considering.

\begin{lemma}
The $(n-1)$-volume of each Voronoi cell of $A^*_{n-1}$ is $\frac{1}{\sqrt{n}}$.
\end{lemma}

\begin{proof}
According to Prop. 2.11 in \cite{baek}, the $d$-volume of a Voronoi cell of $A^*_{d}$ that has been scaled by a factor $d+1$ is $(d+1)^{d - \frac{1}{2}}$, and the result follows. An alternative approach is to note that $A^*_{n-1}$ is a projection of $\mathbb{Z}^n$, in which each Voronoi cell has $n$-volume 1, to $H_{n-1}$. The points that are mapped to the same point in the projection lie uniformly distributed on a line with a distance $\sqrt{n}$ between consecutive points.
\end{proof}

\begin{corollary}
The proportion of $n$-space that is covered by the arrangement is $(n+1) \sqrt{n} \lambda (V)$. {\hfill \ensuremath{\Box}}
\end{corollary}

\begin{example}
For $n=4$, $V$ is given by $x_i - x_j \leq \frac{1}{3}$ for all $i \neq j$. It follows that $V$ is the rhombic dodecahedron with vertices given by the permutations of $\pm \left( \frac{1}{4}, -\frac{1}{12}, -\frac{1}{12}, -\frac{1}{12} \right)$ and $\left( \frac{1}{6}, \frac{1}{6}, -\frac{1}{6}, -\frac{1}{6} \right)$, which has volume $\frac{2}{27}$. Thus, $\frac{20}{27} \approx 0.74074$ of hyperspace is covered.
\end{example}

\begin{example}
For $n=5$, $V$ is given by $p_1 x_1 + \dotsc + p_5 x_5 \leq \frac{5}{4}$ for all extended cyclic permutations $p$ of $(0, 1, 3, -1, -3)$. It follows that $V$ is a polytope with vertices given by the extended cyclic permutations of $$\pm \left( \frac{1}{3}, -\frac{1}{12}, -\frac{1}{12}, -\frac{1}{12}, -\frac{1}{12}\right)$$ $$\left( \frac{1}{4}, -\frac{1}{4}, -\frac{1}{12}, \frac{1}{6}, -\frac{1}{12} \right)$$ $$\left( \frac{1}{12}, \frac{1}{12}, -\frac{7}{36}, \frac{2}{9}, -\frac{7}{36}\right)$$ $$\left( \frac{11}{52}, \frac{11}{52}, -\frac{9}{52}, -\frac{1}{13}, -\frac{9}{52}\right)$$ $$\left( \frac{13}{76}, \frac{13}{76}, \frac{3}{76}, -\frac{7}{76}, -\frac{11}{38} \right)$$ $$\left( \frac{133}{516}, -\frac{37}{516}, \frac{83}{516}, -\frac{13}{129}, -\frac{127}{516}\right)$$ of hypervolume $\frac{8,093,125}{330,355,584} \sqrt{5}$. Thus, $\frac{40,465,625}{55,059,264} \approx 0.73495$ of 5-space is covered.
\end{example}

\section{Another possible generalization}

Instead of $n+1$ directions in $H_n$, we can take $2^{n-1}$ directions in $\mathbb{R}^n$ in a natural manner when $n \geq 4$. The directions $\{ v^{(i)} \}$ can be given as $(\pm 1, \dotsc, \pm 1, 1)$ (we can assume that the last coordinate is positive), and we give a set of vectors $\{u^{(i)}\}$ as before such that the central lines are given by $L_i = u^{(i)} + \mathbb{Z}^n + \{ tv^{(i)} | t \in \mathbb{R}\}$. For $i \neq j$, let $w$ be given by $w_k = v^{(i)}_k \left(u^{(i)}_k - u^{(j)}_k\right)$ for $1 \leq k \leq n$. This vector is split in two - one component $w'$ of dimension $n'$ corresponding to the indices $k$ for which $v^{(i)}_k = v^{(j)}_k$, and one component $w''$ of dimension $n''$ corresponding to the indices $k$ for which $v^{(i)}_k \neq v^{(j)}_k$. The minimal distance between lines oriented in the directions $v^{(i)}$ and $v^{(j)}$ is then $$\sqrt{\left(D_{n'-1}(T_{n'}(w'))\right)^2 + \left(D_{n''-1}(T_{n''}(w''))\right)^2}$$

We think the max min of the distance between lines in different directions is $\frac{\sqrt{5}}{6}$ for $n=4$, attained by taking $$u^{(i)}=\left(\frac{v^{(i)}_2}{4}-\frac{v^{(i)}_1}{6}, -\frac{v^{(i)}_2\left(3v^{(i)}_3+1\right)}{12}, \frac{v^{(i)}_1}{4}, 0\right)$$ for each $i$. For $n=5$, 6, 7 and 8 we can get at least $\frac{1}{\sqrt{6}}$, $\frac{1}{\sqrt{6}}$, $\frac{\sqrt{3}}{4}$ and $\frac{\sqrt{3}}{4}$ respectively, all attainable with $u^{(i)}_j \in \{0, \frac{1}{2} \}$, but better solutions may well exist.

\end{document}